\documentclass{amsart}
\usepackage{a4wide,amsthm,amssymb,amsrefs}
\usepackage[all]{xy}
\usepackage[utf8]{inputenc}
\newcommand{\cC}{\mathcal{C}}
\newcommand{\cD}{\mathcal{D}}
\newcommand{\id}{\mathrm{id}}
\newcommand{\Z}{\ensuremath{\mathbb{Z}}}   

\newcommand{\Q}{\ensuremath{\mathbb{Q}}}

\newcommand{\Mod}[1]{\mathrm{Mod-}{#1}}
\newcommand{\Hom}{\mathrm{Hom}}
\newcommand{\soc}{\mathrm{soc}}

\newcommand{\modo}[1]{\mathrm{mod-}{#1}^{op}}
\newcommand{\m}{\mathbf{m}}
\newtheorem{thm}{Theorem}[section]
\newtheorem{cor}[thm]{Corollary}
\newtheorem{lem}[thm]{Lemma}
\newtheorem{prop}[thm]{Proposition}

\newtheorem{example}[thm]{Example}
\title{Dual Kasch Rings}
\author{Engin Büyüka\c{s}ık}

\address{Izmir Institute of Technology, Department of Mathematics, 35430, Urla, Izmir, Turkey}
\email{enginbuyukasik@iyte.edu.tr}

\author{Christian Lomp}
\address{Faculty of Science of the university of Porto, Department of Mathematics, Rua Campo Alegre, 687, 4169-007 Porto, Portugal}
\email{clomp@fc.up.pt}

\author{Haydar Baran Yurtsever}
\address{Izmir Institute of Technology, Department of Mathematics, 35430, Urla, Izmir, Turkey}
\email{haydaryurtsever@iyte.edu.tr}

\subjclass[2010]{16D10; 16D40; 16D80; 16E30}

\keywords{Injective module; Kasch ring; quasi-Frobenious ring.}

\begin{document}

	\begin{abstract}It is well known that a ring $R$ is right Kasch if each simple right $R$-module embeds in a projective right $R$-module. In this paper we study the dual notion and call a ring $R$ right dual Kasch if each simple right $R$-module is a homomorphic image of an injective right $R$-module. We prove that $R$ is right dual Kasch if and only if every finitely generated projective right $R$-module is coclosed in its injective hull.  Typical examples of dual Kasch rings are self-injective rings, V-rings and commutative perfect rings. Skew group rings of dual Kasch rings by finite groups are dual Kasch if the order of the group is invertible. Many examples are given to separate the notion of Kasch and dual Kasch rings. It is shown that commutative Kasch rings are dual Kasch, and a commutative ring with finite Goldie dimension is dual Kasch if and only if it is a classical ring (i.e. every element is a zero divisor or invertible).  We obtain that, for a field $k$, a finite dimensional $k$-algebra is right dual Kasch if and only if it is left Kasch. We also discuss the rings over which every simple right module is a homomorphic image of its injective hull, and these rings are termed strongly dual Kasch. 
	\end{abstract}
	
	\maketitle
	
	\section{Introduction}
	
	In his work on Frobenius extension of rings \cite{Kasch}, Friederich Kasch considered rings, which he called $S$-rings, that satisfied the condition that every proper left ideal has a non-zero right annihilator and every proper right ideal has a non-zero left annihilator. In his honour, rings such that every proper right ideal has a non-zero left annihilator are called right Kasch rings, while the left Kasch condition being the analogous condition for proper left ideals.
	
	Right Kasch rings are intimately linked to Frobenius extensions, Frobenius algebras and more generally to quasi-Frobenius rings (QF). Many equivalent conditions for a ring $R$ to be right Kasch are known, e.g. $R$ is right Kasch if and only if its injective hull $E(R)$ is an injective cogenerator in Mod-$R$ if and only if  any simple right  $R$-module embeds into $R$.
	
	Recently C.M.Ringel proved in \cite{Ringel}*{3.7}, that over a finite dimensional Nakayama algebra any simple module is a submodule of an indecomposable module with finite projective dimension, and a factor module of an indecomposable module with finite injective dimension\footnote{a finite dimensional algebras $R$ is Nakayama if every indecomposable module has a unique composition series}. 
	Clearly those algebras would be Kasch if the projective dimension of the indecomposable modules in which simples embed is $0$ and allowing the projective dimension to be non-zero but finite seems to be a generalization of the Kasch condition. In this paper we are considering the dual condition of the Kasch condition, namely that every simple module is a factor module of an injective module, i.e. a module of injective dimension $0$. Self-injective rings and V-rings are obvious examples of dual Kasch rings and it became apparent from recent work by the first author in \cite{coneat}, that the dual Kasch condition can be characterized by saying that the ring $R$ is a coneat submodule of its injective hull $E(R)$ (see Theorem \ref{lem:subinjective}). Under additional conditions on $E(R)$, we show that being Kasch implies being dual Kasch and conversely. We prove that dual Kasch rings are invariant under Morita equivalence, but need not be invariant under factor rings.  Similar, to Shen's result in \cite{shen} showing that a group ring $RG$ is Kasch if and only if $R$ is Kasch and $G$ is a finite group, we show that if $G$ is a finite group whose order is invertible in $R$, then the skew group ring $R*G$ is dual Kasch if $R$ is so. In this case also the fixed ring $R^G$ is dual Kasch. 
	It is shown that, a ring $R$ is right self-injective if and only if $R$ right dual Kasch and $E(R)$ is projective (see Proposition \ref{prop:dualKasch_selfinjective}). As a consequence we obtain that,  $R$ is $QF$ if and only if $R$ is one-sided Noetherian, right dual Kasch and $E(R)$ is projective. We also consider a stronger condition by demanding that every simple module is a factor module of its own injective hull. In proposition \ref{prop:QF_Nakayama} we show that QF-rings are strongly dual Kasch if and only if their Nakayama permutation is the identity. 
	
	The third section is devoted to commutative rings, where we  show that any commutative ring that is Kasch or perfect is dual Kasch. Commutative rings with finite Goldie dimension are dual Kasch if and only if they are classical rings, i.e. rings that are equal to their own total ring of quotients, and for commutative Noetherian rings the Kasch and dual Kasch conditions are  equivalent.
	
	The final section is devoted to finite dimensional algebras, where we use duality to show that a finite dimensional algebra is left Kasch if and only if it is right dual Kasch and where we give an example to show that the dual Kasch condition is not left-right symmetric.
	
	All rings $R$ are associative, unital and non-trivial and modules are usually considered to be unital right $R$-modules. For a ring $R$, the Jacobson radical and the right socle of $R$ is denoted by $Jac (R)$ and $\soc(R)$ respectively. By $E(M)$ we shall denote the injective hull of $M$, where $M$ is a right module.

	\section{Kasch rings and their duals}
	
	In this section, we give several characterizations of (strongly) dual Kasch rings. We show that the ring $R$ is dual Kasch if and only if $R$ is a coclosed submodule of its injective hull. Right self-injective rings are dual Kasch. We prove that, $R$ is self-injective if and only if $R$ is right dual Kasch and $E(R)$ is projective. Right Artinian rings that are dual Kasch are characterized. We obtain that, a $QF$ ring is stronly dual Kasch if and only if its Nakayama permutation is the identity.  
	
	Recall that, ring $R$ is called \emph{right Kasch ring} if every simple right $R$-module embeds into $R$. The first part of the following theorem is well known. We include it for completeness.

	\begin{thm}\label{Kasch} Let $R$ be a ring and $E(R)$ the injective hull of $R$ as right $R$-module.
		\begin{enumerate}
			\item[(1)] The following statements are equivalent:
			\begin{enumerate}
				\item[(a)] Every simple right $R$-module is isomorphic to a submodule of a projective module.
				\item[(b)] Every simple right $R$-module embeds into $R$.
			\end{enumerate}
			\item[(2)] The following statements are equivalent:
			\begin{enumerate}
				\item[(a)] Every simple right $R$-module is isomorphic to a factor module of an injective module.
				\item[(b)] Every simple right $R$-module is a homomorphic image of $E(R)$.
			\end{enumerate}	
		\end{enumerate}
	\end{thm}
	\begin{proof}
		(1)  If $(a)$ holds and $S$ is a simple module isomorphic to a submodule of a projective module $P$, then as projective modules are direct summands of free modules, there exists a free module $R^{(\Lambda)}$ and an embedding $f:S \rightarrow R^{(\Lambda)}$. If $\pi_\lambda: R^{(\Lambda)}\rightarrow R$ denotes the projection onto the $\lambda$-component, for $\lambda\in \Lambda$. Then there must exists at least one such $\lambda$ such that $\pi_\lambda\circ f :S\rightarrow R$ is non-zero. Hence $S$ is isomorphic to a submodule of $R$, i.e. $(b)$ holds. The implication $(b)\Rightarrow (a)$ is trivial.
		
		(2) If $(a)$ holds and $S$ is a simple module isomorphic to a factor module of an injective module $E$, then there exists an epimorphism $f:E\rightarrow S$. Moreover, there exists an element $0\neq x\in E$ with $f(xR)=S$. The homomorphism $g:R\rightarrow E$ with $g(r)=xr$ can be extended to a homomorphism $\overline{g}:E(R)\rightarrow E$ such that $\overline{g}(1)=x$ and therefore $f\overline{g}:E(R)\rightarrow S$ is an epimorphism. This shows $(b)$. The implication $(b)\Rightarrow (a)$ is trivial.
	\end{proof}
	
	Condition (1.b) is precisely the condition that defines a ring to be right Kasch.  Let us call rings that satisfy (2.a) {\bf right dual Kasch rings}. Trivial examples of dual Kasch rings are right $V$-rings, since every simple is its own injective hull and right self-injective rings $R$, since any simple is a homomorphic image of $R$. There exists Noetherian V-rings that are domains but not division rings (Cozzens' example). Hence these rings  are right dual Kasch but not right Kasch.
	We will show in the last section that the usual duality for finite dimensional algebras can be used to show that a finite dimensional algebra is left Kasch if and only if it is right dual Kasch. This also allows to give an example of a left dual Kasch that is not right dual Kasch.
	
	\begin{cor}\label{cor:Morita_dualKasch}
		The Kasch and dual Kasch conditions are invariant under Morita equivalence.
	\end{cor}
	
	\begin{proof}
		Let $R$ and $S$ be Morita equivalent rings, then there exists a progenerator $P$ in $Mod-R$ such that $F=\mathrm{Hom}(P,-)$ is an equivalence between $Mod-R$ and $Mod-S$. In particular $F(X)$ is an injective resp. projective $S$-module if and only if $X$ is an injective resp. projective $R$-module  by \cite{wisbauer}*{46.3(1i)}. Furthermore, $F(X)$ is a simple $S$-module if and only if $X$ is a simple $R$-module by \cite{wisbauer}*{46.3(3ii)}. Moreover, since $F$ is an exact functor, a homomorphism $f:X\rightarrow Y$ is injective/surjective if and only if $F(f):F(X)\rightarrow F(Y)$ is injective/surjective. Hence by Theorem \ref{Kasch} $R$ is right Kasch (resp. right dual Kasch) if and only if $S$ is right Kasch (resp. right dual Kasch).
	\end{proof}
	\medskip
	
	Before characterizing dual Kasch rings we need some more terminology:  Following \cite{coneat}, a submodule $N$ of a  right $R$-module $M$ is said to be coneat in $M$ if for all simple right $R$-modules $S$, any homomorphism $N \to S$ can be extended to a homomorphism $M \to S$.  By \cite{coneat}*{Proposition 2.10}, if $N$ is finitely generated, then $N$ is coneat in $M$ if and only if $N$ is coclosed in $M$, where a submodule $N$ of $M$ is said to be coclosed in $M$ if $N/K \ll M/K$ implies $K=N$ for each submodule $K$ of $N$.
	
	

	
	\begin{thm}\label{lem:subinjective}  The following statements are equivalent for a ring $R$:
		\begin{enumerate}
			\item[(a)] $R$ is right dual Kasch.
			\item[(b)] Every finitely generated projective right $R$-module $P$ is coneat in $E(P)$.
			\item[(c)] $R$ is a coclosed submodule of $E(R)$.
		\end{enumerate}
		If any of these conditions hold, then any projective simple right module is injective.
	\end{thm}
	
	\begin{proof} 
		$(a)\Rightarrow (b)$
		Let $S$ be a simple right $R$-module, $P$ be a projective right $R$-module and $f:P\rightarrow S$ any non-zero homomorphism. Since $R$ is right dual Kasch, there is a nonzero homomorphism $g: E(R) \to S$. Note that, $f$ and $g$ are surjective, because $S$ is simple. Then by the projectivity of $P$, there is a nonzero $h: P \to E(R)$ such that $f=gh$ and by the  injectivity of $E(R)$, there is a homomorphism $\overline{h}: E(P) \to E(R)$ which extends $h$. Then $g\overline{h}(P)=gh(P)=f(P)=S$. Hence $g\overline{h}:E(P)\to S$ is non-zero.
		
		\[
		\xymatrix{ 
			0\ar[r] & P \ar[dr]^f \ar[d]_h \ar[r] & E(P)\ar@{-->}[ld]\ar@{-->}[d]^{g\overline{h}} & \\ 
			&E(R) \ar[r]^g & S\ar[r] &  0 
		}\]
		Hence $P$ is coneat in $E(P)$.
		
		$(b) \Rightarrow (c)$ follows from \cite{coneat}*{Proposition 2.10}.

		$(c)\Rightarrow (a)$ If $R$ is coclosed in $E(R)$, then for any maximal right ideal $M\leq R$ we have that $R/M$ is not small in $E(R)/M$. Since $R/M$ is simple and not small, there exists a submodule $U/M \subsetneq E(R)/M$ such that $R/M \oplus U/M = E(R)/M$. Hence $\mathrm{Hom}(E(R),R/M)\neq 0$.
		\medskip
		
		If $R$ is dual Kasch then for any  simple right $R$-module $S$ there exists a non-zero $f:E(R)\to S$, which must be surjective as $S$ is simple. If $S$ is projective, then there exists $U\subseteq E(R)$ such that  $U\oplus \mathrm{Ker}(f)=E(R)$ showing that $S\simeq U$ is injective.
	\end{proof}

	Recall that a right module $M$ is called \emph{retractable} if $\Hom(M, K)\neq 0$ for every non-zero submodule $K$ of $M$. A right module $M$ is called \emph{coretratable} if $\Hom(M/K,M)\neq 0$ for every proper submodule $K$ of $M$.

	\begin{prop} Let $R$ be a ring with injective hull $E(R)$.
		\begin{enumerate}
			\item  If $E(R)$ is retractable and $R$ is right Kasch, then $R$ is right dual Kasch.
			\item If $E(R)$ is coretractable and $R$ is right dual Kasch, then $R$ is right Kasch.
		\end{enumerate}
	\end{prop}
	
	\begin{proof} 
		(1) Let $S$ be a simple right $R$-module. Then $S$ can be embedded in $R$, hence in $E(R)$ and by  assumption  $\Hom(E(R), \,S) \neq 0$. Thus $S$ is a homomorphic image of $E(R)$, and so $R$ is right dual Kasch. 
		
		(2) Let $S$ be a simple right $R$-module. Since $R$ is right dual Kasch, there exists an epimorphism $f: E(R) \to S$. By assumption,  $\Hom (E(R)/Ker f,\,E(R))\neq 0$. Thus $E(R)$ contains a copy of the simple module $E(R)/Ker f \cong S$. Since  $R$ is essential in $E(R)$, $S$ embeds into $R$.
	\end{proof}

	In Corollary \ref{cor:duality_Kasch} we will show that finite dimensional algebras are dual Kasch on one side if and only if they are dual Kasch on the other side and conclude that there are left dual Kasch rings that are not right dual Kasch (see Example \ref{exa:leftNotRight}). We will also see that any commutative perfect ring is dual Kasch. Hence, commutative Artinian rings are dual Kasch. There are one-sided Artinian rings that are not dual Kasch. Here we would like to mention a criterion for a general one-sided Artinian ring to be dual Kasch.
	
	\begin{prop} Let $R$ be a right Artinian ring with Jacobson radical $J=Jac(R)$ and decomposition $R=e_1 R \oplus \cdots e_nR$ with each $e_iR$ is a local right $R$-module\footnote{$e_iR$ is local means that  $e_iJ$ is the unique maximal submodule of $e_iR$}. Then $R$ is right dual Kasch if and only if $\Hom_R(E(e_iR), \, e_iR/e_i J)\neq 0$ for each $i=1,\cdots , n.$
	\end{prop}
	
	\begin{proof} Suppose $R$ is right dual Kasch and let $J=Jac(R)$ Then, by Theorem \ref{lem:subinjective}, $e_iR$ is coneat in its injective hull $E(e_iR)$, for each $i=1,\cdots ,n$. Thus the map 
		$$\Hom(E(e_iR),\, e_iR/e_iJ) \to \Hom(e_iR,\, e_iR/e_iJ) \to 0$$ is surjective. Since $\Hom(e_iR,\, e_iR/e_iJ)\neq 0$, we have $\Hom(E(e_iR),\, e_iR/e_iJ)\neq 0$. This proves the necessity. The sufficiency is clear, because each simple right $R$-module is isomorphic to $e_iR/e_iJ$ for some $i=1,\cdots, n$.
	\end{proof}
	
	Let $R$ denote the ring of upper triangular $n\times n$ matrices over a field $K$, with $n\geq 2$. Then $R=\bigoplus_{i=1}^n e_iR$, where $e_i$ is the matrix that has $1$ in position $(i,i)$ and zero elsewhere.  Then $e_nR=e_nK$ is simple and projective and $e_1R = E(e_nR)$ is the injective hull of $e_nR$. As  
	$\Hom(E(e_nR), e_nR/e_nJ) = \Hom(e_1R, e_nR)=0$, $R$ is not right dual Kasch.
	
	\medskip

	
	While for a dual Kasch ring $R$, any simple module is a homomorphic image of $E(R)$ one might wonder whether they are also homomorphic images of their own injective hull.

	\begin{prop} The following conditions are equivalent for a ring $R$.
		\begin{enumerate}
			\item[(a)]  Every simple right module $S$ is a homomorphic image of its injective hull $E(S)$. 
			\item[(b)] $E(S)$ is retractable for every simple right $R$-module $S$.
		\end{enumerate}
		If one of these conditions hold, $R$ is called {\bf right strongly dual Kasch}.
	\end{prop}

	\begin{proof} $(a)\Rightarrow (b)$  Let $S$ be a simple right $R$-module and $K$ a nonzero submodule of $E(S)$. Since $S$ is essential in $E(S)$, $S$ is contained in $K$. By $(a)$, $\Hom(E(S), S)\neq 0$, which implies that, $\Hom(E(S), K)\neq 0$. Thus $E(S)$ is retractable.
		
		$(b)\Rightarrow (a)$ is trivial
	\end{proof}
	Obviously, strongly dual Kasch rings are dual Kasch and Example \ref{exa:dualnotstrongly} will provide an example of a dual Kasch ring that is not strongly dual Kasch. 
	
	\begin{prop}
		The class of strongly dual Kasch rings is Morita invariant.
	\end{prop}
	\begin{proof}
		Let $R$ and $S$ be rings, $F:Mod-R \rightarrow Mod-S$ a Morita equivalence with inverse $G:Mod-S\rightarrow Mod-R$ such that $S$ is right dual Kasch, then for any simple right $R$-module $X$, let $E_S(F(X))$ be the injective hull of the simple right $S$-module $F(X)$. Let $i:F(X)\rightarrow E_S(F(X))$ be a essential monomorphism. Since $S$ is right dual Kasch, there exists a non-zero homomorphism of right $S$-modules $f:E_S(F(X))\rightarrow F(X)$. By \cite{wisbauer}*{46.3}, $G(E_S(F(X))$ is an injective right $R$-module and $G(i): X\simeq G(F(X)) \rightarrow G(E_S(F(X)))$ is an essential monomorphism. Hence $G(E_S(F(X))) =: E_R(X)$ is an injective hull of $X$. Furthermore, the non-zero homomorphism $f:E_S(F(X))\rightarrow F(X)$ yields a non-zero homomorphism $G(f):E_R(X) = G(E_S(F(X))) \rightarrow G(F(X))\simeq X$. This shows that $R$ is right strongly dual Kasch.
	\end{proof}

	Any commutative perfect ring is actually strongly dual Kasch as we will prove now. Recall that a ring $R$ is called a right $H$-ring if $\mathrm{Hom}(E(S),E(T))=0$, for any non-isomorphic simple right $R$-modules $S,T$. A right max-ring is a ring such that any non-zero module contains a maximal submodule.
	
	\begin{prop}\label{prop:Hmax-strdualKasch} If $R$ right $H$-ring and right max-ring, then $R$ is strongly dual Kasch. In particular any commutative perfect ring is strongly dual Kasch. \end{prop}
	
	\begin{proof}Let $S$ be a simple right module. Then $E(S)$ has a maximal submodule, say $K$, by the max-ring assumption. Then $E(S)/K \cong S$, because $R$ is a right $H$-ring. Hence, $R$ is strongly dual Kasch.
		Commutative perfect rings are max rings. By  \cite{camillo}*{Proposion 2}
		they are also $H$-rings. Hence the conclusion follows.
	\end{proof}

	In general we have the implication
	
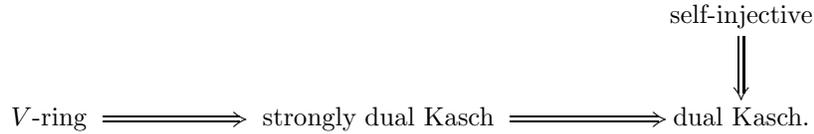
\[
	\xymatrix{
		&&&& \mbox{ self-injective }\ar@{=>}[d]\\
		\mbox{$V$-ring } \ar@{=>}[rr] &&  \mbox{ strongly dual Kasch } \ar@{=>}[rr] && \mbox{dual Kasch.}
	}
	\]
	
	All of these implications are proper. How far is a strongly dual Kasch ring from being a $V$-ring? 
	Note that every projective simple right $R$-module $V$ over a right dual Kasch ring $R$ must be injective. Recall that a right generalized $V$-ring (for short right GV-ring) is a ring such that any simple is injective or projective. Hence right $V$ rings are precisely the right $GV$-rings that are dual Kasch. 
	
	\begin{prop}\label{prop:hereditary} A right dual Kasch ring $R$ is a right $V$-ring if it is right hereditary or a right GV-ring.
	\end{prop}
	
	\begin{proof} Let $S$ be a simple right module. Then $S$ is a homomorphic image of an injective module. Thus $S$ is injective if $R$ is right hereditary assumption. Hence $R$ is right $V$-ring.
		
		By Theorem \ref{lem:subinjective}, projective simple modules over a dual Kasch ring $R$ are injective. Hence if $R$ is a $GV$ ring, it must be a $V$ ring.
	\end{proof}

	It is known that for a finite group $G$ and a ring $R$, the group ring $RG$ is right self-injective (QF) if and only if $R$ is right self-injective (QF) (see for example \cite{Lam}*{Exercicse 15.14} or \cite{jain}, \cite{connell}).
	In \cite{shen} it has been proved that a group ring $RG$ is left Kasch if and only if $R$ is left Kasch and $G$ is finite. Here we prove a similar statement for rings that are right dual Kasch.
	Recall that one says that a group $G$ acts on a ring $R$ by automorphism if there exists a group homomorphism $G\to \mathrm{Aut}(R)$ to the group of automorphisms of $R$. Hence each element $g\in G$ acts on an element $a\in R$ as an automorphism and we denote this action by $g\cdot a$. The skew group ring of $R$ and $G$ is the free left $R$-module with basis $\{ \overline{g} : g\in G\}$ and denoted by $R*G$. Multiplication in $R*G$ is determined by $$a\overline{g} b\overline{h} := a (g\cdot b) \overline{gh}, \qquad \forall a,b\in R, g,h\in G.$$
	A group ring $RG$ is a skew group ring, where the action of an element $g$ on an element $a$ is always trivial, i.e. $g\cdot a = a$. 
	Note that $\imath: R \to R*G$ with $r\mapsto r\overline{e}$ is an injective ring homomorphism, where $e$ is the neutral element of the $G$. Moreover, $R*G = \bigoplus_{g\in G} R\overline{g}$ is a finite normalizing extension of $R$ by which we mean that the $R$-module generators $\overline{g}$ satisfy $R\overline{g} = \overline{g}R$, since $\overline{g}a = (g^{-1}\cdot a) \overline{g}$, for all $a\in R$.
	
	\begin{prop}\label{prop:skewgroup}
		Let $G$ be a finite group acting on a ring $R$ such that $|G|$ is invertible in $R$. If $R$ is right dual Kasch, then $R*G$ is right dual Kasch.
	\end{prop}
	\begin{proof}
		The proof relies on several (well-known) results in the theory of normalized ring extension. Set $S=R*G$. Let $M$ be a simple right $S$-module. Then by \cite{McConnellRobson}*{10.1.9}, $M$ is a semisimple Artinian right $R$-module, as $S$ is a finite normalizing extension of $R$. Since $R$ is right dual Kasch, there exists an injective right $R$-module $E$ and an epimorphism of right $R$-modules $E\rightarrow M$.
		Since $S$ is a free right $R$-module, the functor $\Hom_R(S,-)$ is exact and we obtain an epimorphism of Abelian groups
		$$\overline{f}=\Hom_R(f,-): \Hom_R(S,E) \rightarrow \Hom_R(S,M), \qquad \varphi \mapsto f\circ \varphi.$$
		For any right $R$-module $X$, the set of right $R$-linear maps $\Hom_R(S, X)$ is a right $R$-module by 
		$$ (\varphi\cdot s)(s') = \varphi(ss'), \qquad \forall s,s'\in S, \varphi \in \Hom_R(S,X).$$
		In particular, $\overline{f}=\Hom_R(f,-)$ is an epimorphism of right $S$-modules, as 
		$$\overline{f}(\varphi\cdot s)(s') = f(\varphi(ss'))  = (\overline{f}(\varphi)\cdot s)(s'), \qquad \forall s,s'\in S, \varphi \in \Hom_R(S,X).$$
		
		By \cite{Lam}*{3.6.13}, $\Hom_R(S, E)$, is an injective right $S$-module and we have obtained so far an epimorphism of the injective right $S$-module $\Hom_R(S, E)$ onto $\Hom_R(S, M)$.
		The simple right $S$-module $M$ embeds as right $S$-module into $\Hom_R(S,M)$ by 
		$$\alpha: M \rightarrow \Hom_R(S,M), \qquad m \mapsto \alpha(m):[s\mapsto ms].$$
		We easily check that $\alpha$ is injective and right $S$-linear, as 
		$$\alpha(ms)(s') = mss' = \alpha(m)(ss')=(\alpha(m)\cdot s)(s'), \qquad \forall s,s' \in S.$$
		The map $\beta:\Hom_R(S,M)\rightarrow M$ given by $\beta(f) = f(1)$ is right $R$-linear as 
		$$\beta(f\cdot r) = (f\cdot r)(1)=f(r)=f(1)r=\beta(f)r, \qquad \forall r\in R, f\in \Hom_R(S,M)$$
		and satisfies $\beta(\alpha(m))=\alpha(m)(1)=m$, for all $m\in M$. Thus $\alpha$ splits as $R$-module map.
		
		The hypothesis $|G|^{-1}\in R$ implies that any short exact sequence of right $S$-modules that splits as right $R$-modules, also splits as right $S$-modules (see for example \cite{wisbauer}*{38.4} or \cite{passman}*{Lemma 1.1}). Hence there exists a right $S$-linear map $\overline{\beta}: \Hom_R(S,M)\rightarrow M$ such that $\overline{\beta}\alpha = id_M$. In particular, $\overline{\beta}:\Hom_R(S,M)\rightarrow M$ is an epimorphism of right $S$-modules and so is $\overline{\beta}\overline{f}:\Hom_R(S,E)\rightarrow M$.
		
		This shows that every right $S$-module is an epimorphic image of an injective right $S$-module, i.e. $S=R*G$ is right dual Kasch.
		
	\end{proof}
	
	It is not clear whether the condition of $|G|$ being invertible in $R$ is necessary and we suspect that it is not. The construction of skew group rings allows us to find more (non-commutative) examples of right dual Kasch rings. Let $D$ be any right dual Kasch ring such that $2$ is invertible in $D$. The direct product $R=D\times D$ is also right dual Kasch. Consider the automorphism $\sigma$ of $R$ defined as $\sigma(a,b)=(b,a)$, for all $a,b\in D$ and consider the group $G$ generated by $\sigma$ in $\mathrm{Aut}(R)$. Then $|G|=2$ is invertible in $R$ by assumption and $R*G$ is right dual Kasch. As a concrete example one could take $D=\Q[x]/\langle x^2\rangle$. Then $R*G$ is actually isomorphic to the example mentioned later in Example \ref{exa:dualnotstrongly}.
	
	\medskip
	
	Recall that for a given group action of a group $G$ on a ring $R$, the fixed ring of $G$ on $R$ is the subring
	$$R^G = \{ a\in R: g\cdot a = a, \: \forall g\in G\}.$$
	It is known, that if $|G|^{-1} \in R$, then $R^G$ is Morita equivalent to $R*G$ (see for example \cite{McConnellRobson}*{7.8.7}). Using Proposition \ref{prop:skewgroup} and the Morita invariance of dual Kasch rings in  Corollary \ref{cor:Morita_dualKasch} we conclude the the following:
	
	\begin{cor} Let $G$ be a finite group acting as automorphisms on a ring $R$ such that $|G|^{-1}\in R$. If $R$ is right dual Kasch, then the fixed ring $R^G$ is right dual Kasch.
	\end{cor}
	
	If the group action is trivial, i.e. $R=R^G$, then $R$ and $R*G=RG$ are Morita equivalent if $|G|^{-1}\in R$. Hence we have also
	\begin{cor} Let $G$ be a finite group, $R$ a ring such that $|G|^{-1}\in R$. Then $R$ is right dual Kasch if and only if $RG$ is right dual Kasch.
	\end{cor}
	
	
	How far is a dual Kasch ring from being self-injective? A right $R$-module $M$ is said to be $R$-projective if every homomorphism $M \to R/I$, where $I$ is a right ideal of $R$, factors through the canonical epimorphism $R\to R/I$. If one restricts this property to maximal right ideals $I$ of $R$, then one gets the weaker notion of max-projectivity (see, \cite{maxprojective}).

	\begin{prop}\label{prop:dualKasch_selfinjective} The following conditions are equivalent for a ring $R$:
		\begin{enumerate}
			\item[(a)] $R$ is right self-injective.
			\item[(b)] $R$ is right dual Kasch and $E(R)$ is projective.
		\end{enumerate} 
		Moreover, if $R$ is semilocal, then the following condition is also equivalent:
		\begin{enumerate}
			\item[(c)] $R$ is right dual Kasch and 
			$E(R)$ is max-projective.
		\end{enumerate}
	\end{prop}
	
	\begin{proof} $(a)\Rightarrow (b) \Rightarrow (c)$ are clear.
		
		$(b)\Rightarrow (a)$ If $R$ is right dual Kasch, then $E=E(R)$ generates all simple right $R$-module. Hence $E$ is a projective generator and there exists a surjective homomorphism $f:E^k \to R$, which must split as $R$ is projective. Thus $R$ is isomorphic to a direct sum of $E$ and therefore an injective right $R$-module.
		
		$(c)\Rightarrow (a)$ Assume $R$ is semilocal and $(c)$ holds. Since $R$ is semilocal, $R/Jac(R) \simeq S_1^{n_1} \oplus \cdots \oplus S_k^{n_k}$, where the $S_i $'s are nonisomorphic simple right $R$-modules and $n_i$ is a positive integer for each $i=1,\cdots, k$. Let $E=E(R)$. As the ring is right dual Kasch, for each $i$, there is an epimorphism $E^{n_i} \to S_i^{n_i}$. Thus there is an epimorphism $f: E^n \to R/Jac(R)$ for some integer $n \geq n_1 + \cdots +n_k$. Since $E$ is max-projective, $E^n$ is max-projective as well. As $R/J(R)$ is finitely generated and semisimple, by \cite{maxprojective}*{Proposition 1} there is a homomorphism $g: E^n \to R$ such that $\pi g =f$, where $\pi: R\to R/Jac(R)$ is the canonical epimorphism. Since $f$ is an epimorphism, and $\pi$ is a small epimorphism, $g$ is an epimorphism. Therefore $g$ splits, and so $R$ is injective as desired, i.e. $(a)$ holds.
	\end{proof}
	
	
	A ring is quasi-Frobenius (QF) if it is one-sided Noetherian and one-sided self-injective ring. Any QF ring is indeed two-sided artinian and two-sided self-injective. Hence QF-rings are left and right dual Kasch rings.

	\begin{cor} The following statements are equivalent for a ring $R$.
		\begin{enumerate}
			\item[(a)] $R$ is a QF ring.
			\item[(b)] $R$ is one-sided Noetherian, right dual Kasch and $E(R)$ is projective.
			\item[(c)] $R$ is one-sided Noetherian, right dual Kasch, semilocal and $E(R)$ is max-projective.
		\end{enumerate}
		If moreover, $R$ is commutative then the following condition is also equivalent:
		\begin{enumerate}
			\item[(d)] $R$ is perfect and $E(R)$ is max-projective.
		\end{enumerate}
	\end{cor}
	\begin{proof}
		$(a) \Rightarrow (c), (d)$ is clear and $(a) \Leftrightarrow (b)$ follows from Proposition \ref{prop:dualKasch_selfinjective}.

		$(c)\Rightarrow (a)$  By Proposition \ref{prop:dualKasch_selfinjective}, $R$ is right self-injective and it is well known that,  every one sided noetherian and one sided self-injective ring is $QF$. 
		
		$(d) \Rightarrow (a)$ Suppose $R$ is a commutative perfect ring such that $E(R)$ is max-projective, then it is dual Kasch by Proposition \ref{prop:Hmax-strdualKasch}  and self-injective by Proposition \ref{prop:dualKasch_selfinjective}. By 
		\cite{qf}*{Theorem 6.39} any   left perfect two sided self-injective ring is a QF-ring. 
	\end{proof}
	

	Not every QF ring is strongly dual Kasch as we will see. For that, we recall the notion of the Nakayama permutation of a QF ring (see, \cite{Lam}*{page 425}). Let $R$ be a QF ring and $\bar{R}=R/Jac(R)$. Then we have 
	$$1=e_{11} + \cdots +e_{1n_1} + \cdots + e_{s1} + \cdots + e_{sn},$$be a decomposition of 1 into sum of orthogonal primitive idempotents, where $e_i :=e_{i1}$ for $1 \leq i \leq s$ are mutually nonisomorphic i.e. $e_iR \ncong e_j R$ for $i \neq j$, but $e_i$ is isomorphic to each $e_{il}$.  Let $$U_i =e_iR,\,\,\,S_i= \bar{e_i}\bar{R}.$$Then, $\{ U_i \}$ is a complete set of right principal indecomposables, and $\{S_i \}$ is a complete set of simple right $R$-modules. Then  $R_R=  U_1^{n_1} \oplus \cdots  \oplus U_s^{n_s }$ and $\bar{R}_{\bar{R}}=S_1^{n_1 } \oplus \cdots  \oplus  S_s^{n_s}.$ Note that, $S_i \cong U_i/Rad (U_i)$. Since QF rings are right Kasch, the map $\pi :\{1,\cdots ,s\} \to \{1,\cdots ,s \},$ defined by $$\soc(U_i) \cong S_{\pi (i)},\,\,\, 1 \leq i \leq s$$
	is a permutation of the set $\{1,\cdots, s\}$. This map $\pi$ is called the Nakayama permutation of the QF ring $R$. 
	
	\medskip

	Note that each $U_i$ is injective, and is the injective hull of $\soc(U_i)$.  Furthermore, $\pi$ is the identity permutation on the set $\{1,2,\cdots ,s\}$ if and only if $\soc(U_i)\cong S_{i}$ if and only is $Hom (U_i, \, S_i)\neq 0$. Hence, we have the following result.
	
	\begin{prop}\label{prop:QF_Nakayama} A $QF$ ring is strongly dual Kasch if and only if its Nakayama permutation is the identity permutation.
	\end{prop}
	
	
	
	
	A QF algebra is called \emph{weakly symmetric} if its Nakayama permutation $\pi$ is the identity (see \cite{Lam}*{page 444}). Hence the following is clear.

	\begin{prop} A $QF$ algebra is strongly dual Kasch if and only if it is weakly symmetric.
	\end{prop}
	
	Not all $QF$-algebras over a field are weakly symmetric. The following example is  \cite{Lam}*{Examples 16.9(4)} with a different presentation:
	
	\begin{example}[A dual Kasch ring that is not strongly dual Kasch]\label{exa:dualnotstrongly}
		Let $F$ be a field and $S=F\times F$. Set $e_1=(1,0)$ and $e_2=(0,1)$ and define the automorphism
		$\sigma:S\rightarrow S$ given by $\sigma(e_1)=e_2$ and $\sigma(e_2)=e_1$. Then form the skew polynomial ring $S[x;\sigma]$, i.e. the set of all polynomials $\sum s_ix^i$ with $s_i\in S$ subject to 
		$ xs = \sigma(s)x$. In particular, $x e_1 = e_2 x$ and $xe_2 = e_1x$ hold. Factoring out the ideal generated by the central element $x^2$ yields a $4$-dimensional $F$-algebra
		$$ R = S[x; \sigma]/\langle x^2\rangle = \{ s_0 + s_1x \mid s_0,s_1 \in S\}$$
		where we write $x$ for $x+\langle x^2\rangle$ and where the multiplication is 
		$$ (s_0 + s_1x)(t_0 + t_1x) = s_0t_0 + \left(s_0t_1 + t_0\sigma(s_1)\right)x, \qquad \forall s_i, t_i \in S.$$
		Then $ R = e_1 R \oplus e_2 R$, $J=\mathrm{Jac}(R)=\soc(R_R) = Fe_1x \oplus Fe_2x$.
		Hence $e_iJ = \soc(e_1R)=Fe_ix$, for $i=1,2$ and  $e_1J \not\simeq e_2J$.
		On the other hand note that
		$e_1R/e_1J \simeq e_2J$  and $e_2R/e_2J\simeq e_1J$. 
		By  \cite{Lam}*{Theorem 16.4}, $R$ is QF with Nakayama permutation $\pi=(12)$. In particular, $R$ is not weakly symmetric and hence not strongly dual Kasch. As a two-sided self-injective ring, $R$ is two-sided dual Kasch.
	\end{example}

	



	\section{Commutative Dual Kasch Rings}
	In this section, we deal with (strongly) dual Kasch rings over commutative rings. We characterize commutative dual Kasch rings with finite Goldie dimension.
	We start with the following Lemma which implies that commutative Kasch rings are  (strongly) dual Kasch.

	\begin{lem}\label{lem:simplequotient} Let $V$ be a  minimal ideal in a commutative ring $R$. Then $V$ is a homomorphic image of $E(V)$.
	\end{lem}
	
	\begin{proof} Let $V$ be a minimal ideal of $R$ with annihilator $\mathfrak{m}=\{r\in R: Vr=0\}$.  Let $E=E(V)$ be the injective hull of $V$. Then $E(R)=E \oplus K$ for some submodule $K$ of $E(R)$. Suppose $E=\mathfrak{m}E$, then $V E=V \mathfrak{m}  E=0$. Since $1 \in E(R)$, there are $v \in E(V)$ and $k \in K$ such that $1=v+k$. Let $0 \neq x \in V$. Then $x=x\cdot 1 =x\cdot v + x\cdot k=x\cdot k$. Thus $x \in E(V) \cap K=0$, a contradiction. Therefore, we must have $E \neq \mathfrak{m} E$. Thus  $E/\mathfrak{m} E$ is a semisimple $R/\mathfrak{m}$-module and hence there exists a non-zero homomorphism from $E(V)$  to $R/\mathfrak{m}\simeq V$. 
	\end{proof}
	
	\begin{cor} Every commutative  Kasch ring is strongly dual Kasch.
	\end{cor}
	
	
	
	
	\medskip
	
	Kasch rings are related to classical rings, i.e. rings such that any element is either a zero divisor or invertible. Recall that a \emph{reversible} ring is a ring such that for all non-zero $a,b\in R$, if  $ab=0$, then there exists $0\neq c\in R$ such that $bc=0$.  
	
	\begin{cor}\label{cor:reversibleKasch}
		Any reversible right Kasch ring  is classical.
	\end{cor}
	
	\begin{proof}
		By \cite{Lam}*{8.28} we have that $R$ is a right Kasch ring if  and only if every  proper right ideal has a non-zero left annihilator. Hence if $0\neq a\in R$ and $aR\neq R$, then there exists $0\neq x\in R$ with $xa=0$. By assumption there exists $y\in R$ such that $ay=0$. On the other hand, if $xR=R$, then $x$ would be right invertible and since reversible rings are Dedekind-finite, $x$ is invertible. Hence any element of $R$ is either a zero divisor or invertible.
	\end{proof}
	
	Note that any left or right Artinian ring is reversible and classical, but need not to be Kasch.  For example the ring of upper triangular $2x2$ matrix ring over a field is finite dimensional algebra, which is not a Kasch ring. However for commutative rings with some finiteness conditions, the notions of a Kasch ring, a dual Kasch ring and a  classical ring coincide as we will see, using   \cite{Filipowicz}*{Theorem 3.4}, which says that any finitely generated ideal of a commutative ring with finite Goldie dimension has a non-zero annihilator.
	
	\begin{thm}\label{thm:dualKasch_comm} The following are equivalent for a commutative ring R with finite Goldie dimension.
		\begin{enumerate}
			\item[(a)]  $R$ is a dual Kasch ring.
			\item[(b)]  $R$ is a classical ring.
			\item[(c)]  $\mathfrak{m}E \neq E$ for each maximal ideal $\mathfrak{m}$ of $R$, where $E=E(R)$.
		\end{enumerate}
		If moreover $R$ is Noetherian, then the following condition is also equivalent:
		\begin{enumerate}
			\item[(d)] $R$ is a Kasch ring.
		\end{enumerate}
	\end{thm}

	\begin{proof} 
		$(a) \Rightarrow (b)$ Let $x$ be a non zero divisor of $R$. Then $x$ is a regular element of $R$.  Suppose $x$ is non invertible, and let $\mathfrak{m}$ be a maximal ideal of $R$ containing $x$.  Since $x$ is  regular $xE=E$ for every injective $R$-module $E$. This implies that $\mathfrak{m}E=E$ for every injective module $E$. Then $\Hom (E, R/\mathfrak{m})=0$, and this contradicts (1). Hence $R$ is a classical ring.
		
		\medskip
		
		$(b) \Rightarrow (c)$ Suppose the contrary that  $\mathfrak{m}E = E$ for some maximal ideal $\mathfrak{m}$ of $R$. Then as $R \leq E$, there are $a_1,\,a_2,\cdots a_n \in \mathfrak{m}$ and $e_1,\,e_2,\cdots e_n \in E$ such that $$1=a_1e_1 +a_2 e_2 + \cdots + a_n e_n.$$ Let $I=Ra_1 +Ra_2 + \cdots Ra_n.$ Then $I \subseteq \mathfrak{m}$ is a proper ideal of $R$, and so every element of $I$ is a zero divisor by $(2)$. Thus there is a $0 \neq x \in R$ such that $xI=0$ by \cite{Filipowicz}*{Theorem 3.4}. Which then implies that $$x=x(a_1e_1 +a_2 e_2 + \cdots + a_n e_n)=(xa_1)e_1 +(xa_2) e_2 + \cdots + (xa_n) e_n=0.$$A contradiction. Therefore $\mathfrak{m}E \neq E$, and this proves $(3)$.
		
		\medskip
		
		$(c) \Rightarrow (a)$ is clear, since $R$ is commutative. As $\mathfrak{m}E\neq E$ implies that $E/\mathfrak{m}E$ is a semisimple $R/\mathfrak{m}$-module and hence projects onto $R/\mathfrak{m}$.
		
		\medskip
		
		$(d)\Rightarrow (a+b+c)$ follows from Corollary \ref{cor:reversibleKasch} and holds always for commutative rings.
		
		$(b)\Rightarrow (d)$: Suppose $R$ is a Noetherian classical ring and $I$ any proper  ideal of $R$, then $I$ is finitely generated as $R$ is Noetherian. Since a Noetherian ring has finite Goldie dimension, $I$ has non-zero annihilator by \cite{Filipowicz}*{Theorem 3.4}. Hence, by \cite{Lam}*{8.28}, $R$ is a Kasch ring.
	\end{proof}
	
	\begin{example}[dual Kasch rings that are not Kasch]
		In general, commutative dual Kasch rings need not be Kasch rings. Let $R$ be a commutative von Neumann regular ring which is not semisimple. Then $R$ dual Kasch, because every every simple $R$-module is injective. On the other hand, $R$ is not a Kasch ring since $soc(R)$ is projective and $R$ is not semisimple, i.e. there exist non-projective simple modules.
	\end{example}

	Commutative local rings need not be dual Kasch. For example, the localization  $\Z_{(p)}$  of $\Z$ at the prime ideal $(p)$ is a commutative Noetherian local ring, hence semiperfect. Since $E(\Z_{(p)})=\Q$ has no maximal submodules, the ring  $\Z_{(p)}$ is not dual Kasch.
	
	\begin{example}[dual Kasch rings are not closed under factor  rings]
		Let $R$ be any commutative Noetherian local ring with maximal ideal $M$ and zero socle. 
		Form the trivial extension of $R$ and $S=R/M$, i.e.
		$$\tilde{R}=R\ltimes S = \left\{ \left(\begin{array}{cc} a & x \\ 0 & a\end{array}\right) : a\in R, x\in S\right\}$$ 
		with ordinary matrix operations. Then $\tilde{R}$ is a commutative local ring with maximal ideal $\tilde{M} = M\ltimes S$.  As $\tilde{R}/\tilde{M} \simeq R/M  = S \simeq 0\ltimes S$, $\tilde{R}$ is a commutative local Kasch ring and hence dual Kasch.
		
		\medskip
		
		For each essential ideal $I$ of $R$, the ideal 
		$$\tilde{I} =  \left\{ \left(\begin{array}{cc} a & x \\ 0 & a\end{array}\right) : a\in I, x\in S\right\}$$ 
		is an essential ideal of $\tilde{R}$. Since $\soc(R)=0$, the intersection of all essential ideals of $R$ is zero. Hence 
		$$\soc(\tilde{R}) = \bigcap \left\{ \tilde{I} : I \mbox{ essential in } R\right\} = \left(\begin{array}{cc} 0 & S \\ 0 & 0\end{array}\right) = 0\bowtie S.$$
		As $\tilde{R}/\soc(\tilde{R})  \simeq R$ has zero socle and is therefore not Kasch. As $R$ was supposed to be commutative Noetherian, it cannot be dual Kasch by Theorem \ref{thm:dualKasch_comm}.
	\end{example}

	\section{Dualities and Artin algebras}
	
	The aim of this section is to compare the notion of Kasch and dual Kasch ring in case a duality is present. A standard result of Artin algebras $A$  says that there exists a duality between finitely generated right $A$-modules and finitely generated left $A$-modules (see  \cite{AuslanderReitenSmalo}*{Theorem 3.3}). 
	For us this will mean that an Artin algebra (see definition below) is Kasch on one side if and only if it is dual Kasch on the other side (see Theorem \ref{thm:Auslander}). 
	
	\medskip
	
	Let us recall what a duality is. Two categories $\cC, \cD$ are called equivalent if there are covariant functors $F:\cC \rightarrow \cD$ and $G:\cD\rightarrow \cC$ with functorial isomorphisms $GF \simeq \id_\cC$ and $FG \simeq \id_\cD$. In this case the functors $F$ and $G$ are called \emph{equivalences}. We say that $G$ is the (equivalence) inverse of $F$.
	A duality between two categories $\cC$ and $\cD$ is an equivalence between $F:\cC\rightarrow \cD^{op}$, where $\cD^{op}$ denotes the opposite category. In this case $F$ and $G$ are called \emph{dualities}. 
	
	Let $R$ and $S$ be rings and $\cC\subseteq \Mod{R}$ and $\cD\subseteq \Mod{S}$ be full subcategories of $\Mod{R}$ and $\Mod{S}$ respectively. Suppose there exists a duality $F:\cC\rightarrow \cD$ with inverse $G:\cD\rightarrow \cC$ such that $F$ and $G$ are contravariant exact functors. Then 
	$M\in \cC$ is injective in $\cC$ if and only if  $F(M)\in \cD$ is projective in $\cD$ and  $M\in \cC$ is projective in $\cC$ if and only if $F(M) \in \cD$ is injective in $\cD$. Moreover, $M\in \cC$ is simple if and only if $F(M)\in \cD$ is simple (see for instance \cite{AndersonFuller}*{Exercise 23.6}).
	
	\begin{lem}
		Let $R$ and $S$ be rings and $\cC\subseteq \Mod{R}$ and $\cD\subseteq \Mod{S}$ be full subcategories.
		Suppose there exists a duality $F:\cC\rightarrow \cD$ with inverse $G:\cD\rightarrow \cC$ such that $F$ and $G$ are contravariant exact functors. Then the following are equivalent:
		\begin{enumerate}
			\item[(a)] Every simple object in $\cC$ is isomorphic to a subobject of a projective object in $\cC$.
			\item[(b)] Every simple object in $\cD$ is isomorphic to a quotient object of an injective object in $\cD$.
		\end{enumerate}
		
	\end{lem}	
	
	\begin{proof}
		$(a)\Rightarrow (b)$ If $S\in \cD$ is simple, then $G(S)$ is simple in $\cC$. By (1) there exists a monomorphism $f:G(S)\rightarrow P$ into a projective object $P\in\cC$. Then $F(f):F(P)\rightarrow FG(S)\simeq S$ is an epimorphism of the injective object $F(P)$ onto $S$. The converse $(b)\Rightarrow (a)$ goes analogously.
	\end{proof}
	
	Let $k$ be a field and $R$ a finite dimensional $k$-algebra. Then taking the dual space is a duality between finite dimensional right $R$-modules and finite dimensional left $R$-modules. More concretely, let $M$ be a finite dimensional (=finitely generated) right $R$-module, then the left $R$-module structure on $D(M)=\mathrm{Hom}_k(M,k)$ is defined by
	$$ (a\cdot f)(m) := f(ma), \qquad \forall f\in D(M), m\in M, a\in R.$$
	We will denote use the same letter $D$ for the inverse functor $D:\modo{R}\rightarrow \mod{R}$ sending $D(M)=M^*$. 
	Here  $R^{op}$ denotes the opposite ring and $\modo{R}$ is identified with the category of left $R$-modules. The right $R$-module structure on $D(M)$, for a left $R$-module $M$ is defined by
	$$ (f\cdot a)(m) := f(am), \qquad \forall f\in D(M), m\in M, a\in R.$$
	
	Note that if $R$ is finite dimensional, then all simple left (resp. right) $R$-modules are finite dimensional. Furthermore, any projective module in $\mod{R}$ is also projective in $\Mod{R}$. The same is true for the injective modules, i.e. any finite dimensional module that is injective in $\mod{R}$ is injective in $\Mod{R}$ (by Baer's criterion).

	\begin{cor}\label{cor:duality_Kasch} A finite dimensional $k$-algebra is right dual Kasch if and only if it is left Kasch.
	\end{cor}
	\begin{proof}
		By  \cite{AssemSimsonSkrowonski}*{Theorem 5.13} or \cite{AndersonFuller}*{Exercise 23.6}, the duality functor  $D:\mod{R}\rightarrow \modo{R}$ is an exact duality for a finite dimensional $k$-algebra $R$.
	\end{proof}
	
	\begin{example}[An example of a left dual Kasch ring that is not right dual Kasch]\label{exa:leftNotRight}
		In \cite{Lam}*{8.29(6)} an example of a $5$-dimensional algebra $A$ over a field $K$ is given that is right Kasch, but not left Kasch. Hence by Corollary \ref{cor:duality_Kasch} and its left version, $A$ is left dual Kasch, but not right dual Kasch. The example given in \cite{Lam}*{8.29(6)} is the subring of $M_4(K)$, the ring of $4\times 4$-matrices with coefficients in $K$ of the form
		$$A=\left\{\left(\begin{array}{ccccc} 
			a & 0 & b & c \\
			0 & a & 0 & d \\
			0 & 0 & a & 0 \\
			0 & 0 & 0 & e \end{array}\right) \in M_{4}(K): a,b,c,d,e \in K\right\}$$
		
	\end{example}
	
	\medskip

	The duality for finite dimensional algebras can be generalized to Artin algebras. An Artin algebra is a ring $A$ with a ring  homomorphism $R\rightarrow Z(A)$ from a commutative Artinian ring $R$ into the center of $A$ such that $A$ is a finitely generated $R$-module. Since $R$ is Artinian it has only a finite number of non-isomorphic simple modules, say $S_1, \ldots, S_n$. Let $E(S_i)$ be the injective hull of $S_i$ in $\mod{R}$ and consider $Q=E(S_1)\oplus \cdots \oplus E(S_n)$, which is an injective cogenerator in $\mod{R}$.
	
	\begin{thm}[{\cite{AuslanderReitenSmalo}*{Theorem 3.3}}]\label{thm:Auslander}
		If $A$ is an artin $R$-algebra, then the contravariant exact functor $D:\mod{A}\rightarrow \modo{A}$ with 
		$$D(M) = \mathrm{Hom}_R(M,Q), \qquad \forall M\in \mod{A}$$
		is a duality. In particular $A$ is a right dual Kasch ring if and only if $A$ is left Kasch.
	\end{thm}

	
	

	\section*{Acknowledgment}
	The second author, Christian Lomp, was partially supported by CMUP, which is financed by national funds through FCT – Fundação para a Ciência e
	a Tecnologia, I.P., under the project with reference UIDB/00144/2020. This paper is a part of M.Sc. Thesis of the third author.

\end{document}